\documentclass[12pt, reqno]{amsart}

\usepackage{a4wide}
\usepackage{amsmath,amsfonts,amssymb,
amsthm, color}
\usepackage{hyperref}
\usepackage{csquotes}
\usepackage{dirtytalk}

\newcommand\numberthis{\addtocounter{equation}{1}\tag{\theequation}}

\usepackage[active]{srcltx}
\usepackage{dutchcal}
\usepackage{mathabx}
\usepackage{mathrsfs}
\usepackage{multicol,xcolor}

\newtheorem{theorem}{Theorem}[section]
\newtheorem{definition}[theorem]{Definition}
\newtheorem{proposition}[theorem]{Proposition}
\newtheorem{corollary}[theorem]{Corollary}

\theoremstyle{definition}
\newtheorem{remark}[theorem]{Remark}

\newtheorem{notation}[theorem]{Notation}

\def\R{\mathbb{R}}

\def\M{\mathbb{M}}

 \def\Op{\mathfrak{Op}} 

\def\X{\mathcal X}

\def\g{\mathcal{g}}

\def\Rd{\mathbb{R}^d}
\def\Zd{\mathbb{Z}^d}
\def\Z{\mathbb{Z}}

\def\bb1{{\rm{1}\hspace{-3pt}\mathbf{l}}}
\def\Ie0{[-\epsilon_0,\epsilon_0]}

\def\supp{\text{\sf supp}}

\def\supp{\mathop{\rm supp} \nolimits} 

\def\BC2{\mathcal L \big(\mathbb{C}^2\big)}

\def\MAF{\mathbb{M}^A[F]}

\def\tGamma{\widetilde{\Gamma}}
\def\talpha{\tilde{\alpha}}
\def\tbeta{\tilde{\beta}}
\def\tgamma{\tilde{\gamma}}
\def\tmu{\tilde{\mu}}
\def\tnu{\tilde{\nu}}

\def\beq{\begin{equation}}
\def\eeq{\end{equation}}

\numberwithin{equation}{section}

\title[Magnetic pseudo-differential operators via tight Gabor frames]{Matrix representation of magnetic pseudo-differential operators via tight Gabor frames}
\author{Horia D. Cornean}
\address{Department of Mathematical Sciences, Aalborg University, DK-9220 Aalborg, Denmark. Orcid: 0000-0003-2700-8785}
\email{cornean@math.aau.dk}

\author{Bernard Helffer}
\address{Laboratoire de Math{\'e}matiques Jean Leray, Universit{\'e} de Nantes and CNRS, Nantes, France.}
\email{Bernard.Helffer@univ-nantes.fr}

\author{Radu Purice}
\address{\enquote{Simion Stoilow} Institute of Mathematics of the Romanian Academy, P.O. Box 1-764, 014700 Bucharest, Romania. Orcid: 0000-0003-2700-8785}
\email{Radu.Purice@imar.ro}
\thanks{HC and RP acknowledge support from Grant 8021-00084B of the Independent Research Fund Denmark $|$ Natural Sciences. RP acknowledges partial support by the CNRS IRN ECO-Math. The authors thank M. Lein and G. Lee for very useful comments which significantly improved the manuscript.}

\subjclass[2020]{Primary: 81Q10, 81Q15. Secondary: 35S05}
\keywords{Gabor frames, pseudodifferential operators, magnetic fields.}

\begin{document}
\maketitle
\begin{abstract}
In this paper we use some ideas from \cite{FG-97, G-06} and consider the description of H\"{o}rmander type pseudo-differential operators on $\Rd$ ($d\geq1$), including the case of the  magnetic pseudo-differential operators introduced in \cite{IMP-1, IMP-19}, with respect to a tight Gabor frame. We show that all these operators can be identified with some infinitely dimensional matrices whose elements are strongly localized near the diagonal. Using this matrix representation, one can give short and elegant proofs to classical results like the Calder{\'o}n-Vaillancourt theorem and Beals' commutator criterion, { and also establish local trace-class criteria.} 
\end{abstract}

\section{Introduction}

\subsection{Main goals}
In this work we continue the study of Gabor frame decomposition of Pseudo-Differential Operators (in short $\Psi$DO's) on $\Rd$ ($d\geq1$), as proposed in \cite{FG-97, G-06}, in order to characterize  H\"{o}rmander type $\Psi$DO's with symbols in class $S^p_0(\R^{2d})$, with $p\in\R$, \cite{H-3}, including the \say{magnetic twisted case} introduced in \cite{MP-1,IMP-1,IMP-19}. 
We recall that these are symbols $a(\xi,x)$ such that $(1+\xi^2)^{{ -p/2}}\, a(\xi,x)$ is uniformly bounded, but no decay is gained by differentiation, see Definition \ref{D-Hsymb}. Our main result (Theorem \ref{T-m-Gabor-frame}) shows that their infinitely dimensional matrices in the tight \say{magnetic} Gabor frame \eqref{DF-G-frame} that we consider, are strongly localized around their diagonal, with some precise growth condition.

 This matrix representation leads almost immediately to a short proof for the \say{magnetic} version of the Calder\'on-Vaillancourt Theorem in \cite{IMP-1} (Theorem \ref{T-m-CV-Thm} in this paper) which completes our previous result in \cite{CHP-3} regarding the Beals commutator criterion (Theorem \ref{T-recBeals} in this paper).

 At the same time, we shed new light on some previous results obtained in \cite{CIP, CHP-3, CGSS, CMM} and we include some developments of the ideas which were introduced there. For example, in Corollary \ref{coro1} we obtain in a straightforward way that the magnetic Moyal product (see \eqref{DF-mMprod}) of a symbol of class $S_0^p$ with a symbol of class $S_0^q$ produces a symbol of class $S_0^{p+q}$. { Moreover, in Theorem \ref{thm-trace} we give a short and straightforward proof of the fact that if $p<-d$, then the corresponding pseudo-differential operators are locally trace class. We also hope that our approach could be relevant for the further development of the magnetic super-operator calculus \cite{LG}.}
 
 Let us emphasize that the type of arguments we develop in this paper are related in spirit with the \say{partition of unity} techniques developed in \cite{Bo-96-97, BoCh, BoLe, Le-LN-19}.

\subsection{General notation}

 For  $N\in\mathbb{N}\setminus\{0\}$, let $\mathscr{S}(\R^N)$ be the space of Schwartz test functions on $\R^N$ with the canonical Fr\'{e}chet topology and $\mathscr{S}^\prime(\R^N)$ its topological dual with its strong dual topology and let us denote by $\langle\cdot\,,\,\cdot\rangle_{{\mathscr{S}', \mathscr{S}}}:\mathscr{S}^\prime(\R^N)\times\mathscr{S}(\R^N)\rightarrow\mathbb{C}$ the canonical duality map.
We denote by $\mathcal{L}(\mathcal{V}_1;\mathcal{V}_2)$ the space of linear continuous operators between the topological vector spaces $\mathcal{V}_1$ and $\mathcal{V}_2$ with its strong topology (of uniform convergence on bounded sets). 

For some $d\geq2$ we  consider the $d$-dimensional real affine space $\X$ that we shall freely identify with $\Rd$ considering fixed a \say{base point}. Let $\Xi:=\X\times\X^*$ where $\X^*$ is the dual space of $\Rd$. We shall always distinguish between position variables and momentum variables. We recall the notation $<x>:=\sqrt{1+|x|^2}$ and similarly for $\xi\in\X^*$. We use the notation $\Xi:=\X\times\X^*$.

We shall work with the usual Lebesgue measure on $\X$ and the associated Hilbert space $L^2(\X)$ with the scalar product considered anti-linear in the first factor and denoted by $\big(\cdot,\cdot\big)_{L^2(\X)}$. We notice that:
\beq\label{F-dual-pscal}
\langle\overline{f},\phi\rangle_{{\mathscr{S}', \mathscr{S}}}\,=\,\big(f,\phi\big)_{L^2(\R^d)},\quad\forall(f,\phi)\in L^2(\R^d)\times\mathscr{S}(\R^d).
\eeq

We will use the H\"{o}rmander multi-index notation $\partial^a_x:=\partial_{x_1}^{a_1}\cdot\ldots\cdot\partial_{x_d}^{a_d}$ for $x\in\X$  and $\partial^a_{\xi}:=\partial_{\xi_1}^{a_1}\cdot\ldots\cdot\partial_{\xi_d}^{a_d}$ for $\xi\in\X^*$ and $|a|:=a_1+\ldots+a_d$ for any $a\in\mathbb{N}^d$.\\
Given any measurable function $F:\X\rightarrow\mathbb{C}$ we denote by $F(Q)$ the operator of multiplication with the function $F$.  Let $\mathcal L (\mathcal{H})$ be the $C^*$-algebra of bounded linear operators and $\mathbb{U}(\mathcal{H})$ the group of unitary operators on the complex Hilbert space $\mathcal{H}$. 

Let us fix our notations and normalization for the Fourier transform:
\beq\begin{split}
&\mathcal{F}_\X:L^1(\X)\rightarrow C(\X^*),\quad\big(\mathcal{F}_{\X}f\big)(\xi):=(2\pi)^{-d/2}\int_{\X}dx\,e^{-i<\xi,x>}\,f(x),\\
&\mathcal{F}_{\X^*}:L^1(\X^*)\rightarrow C(\X),\quad\big(\mathcal{F}_{\X^*}\hat{f}\big)(x):=(2\pi)^{-d/2}\int_{\X^*}d\xi\,e^{i<\xi,x>}\,\hat{f}(\xi).
\end{split}\eeq
They have unitary extensions to $L^2$ which are inverse to each other.

\subsection{The magnetic field.} \label{ss1.3}
We shall consider \textit{\say{regular} magnetic fields} $B$, described by smooth closed 2-forms on $\X\cong\Rd$ which have components of class $BC^\infty(\X)$, i.e.:
\begin{align}\label{dhc1}
B:=\sum_{1\leq j,k\leq d} B_{jk}(x) dx_j\wedge dx_k\,, \quad |||B|||_N:=\max_{1\leq j,k\leq d}\Vert B_{jk}\Vert_{BC^N(\X)}<\infty,\, \forall \; N\geq 0.
\end{align}
  The topological triviality of the affine space $\X$ implies that these 2-forms are also exact, hence we can always find a 1-form
  \begin{subequations}\label{dc1'}
   \begin{equation}
   A=\sum_j A_j(x) dx_j
   \end{equation}  on $\X$ such that 
$$B=dA,\mbox{  i.e. } B_{jk}(x)=(\partial_jA_k)(x)-(\partial_kA_j)(x)\,.
$$
This choice is far from being unique and we may consider gauge transformations $A\mapsto A^\prime=A+dF$ with $F\in C^\infty(\X)$ so that $B=dA=dA^\prime$. The following explicit choice:
\beq
A_k(x)\,=\,\underset{1\leq j\leq d}{\sum}\int_0^1\,s\, x_j\, B_{jk}(sx)\,ds \,,
\eeq
\end{subequations}
proves that for any regular $B$ we can always choose its vector potential to have smooth components which grow at most polynomially at infinity.

\section{The Gabor frame}
\subsection{Definition}
We shall consider the lattice $\Gamma=\underset{1\leq j\leq d}{\sum}\Z \mathfrak{e}_j\cong\Zd$ defined by the canonical orthonormal basis $\{\mathfrak{e}_j\}_{1\leq j\leq d}$ of $\X\cong\Rd$ and let $\Gamma_*$ be the dual lattice:
$$
\Gamma_*:=\big\{\gamma^*\in\X^*\,\mbox{ s.t.}\,<\gamma^*,\gamma>\in2\pi\Z\,,\forall\gamma\in\Gamma\big\}
.
$$

Once fixed the lattice $\Gamma\subset\X$ as above, let us choose a quadratic partition of unity associated with  the lattice $\Gamma$, i.e. a function $\mathcal{g}\in C^\infty_0(\X)$ such that
\beq\label{dc3}
\supp\mathcal{g}\subset(-1,1)^d,\quad \underset{\gamma\in\Gamma}{\sum}\,\mathcal{g}(x-\gamma)^2\,=\,1,\quad\forall x\in\X.
\eeq
With any $f\in L^2(\X)$ we associate the $\Gamma$-indexed sequence:
$$
f_\gamma=\mathcal{g}\,\tau_{\gamma}(f)\,,\quad (\tau_\gamma f)(x)=f(x+\gamma), \quad \gamma\in\Gamma.
$$
Each $f_\gamma$ has compact support in $(-1,1)^d\subset (-\pi,\pi)^d$.  On this support it coincides with its own $2\pi \Z^d$ periodization  that we denote by $  \mathring{f}_\gamma$.  We can then define its Fourier sequence:
\beq\label{dc1}
\widehat{[f_\gamma]}_{\gamma^*}\,:=\,\big (\vartheta_{\gamma^*}\, ,\, f_\gamma\big )_{L^2((-\pi,\pi)^d)},\quad \gamma^*\in\Gamma_*,\quad \vartheta_{\gamma^*}(x):=(2\pi)^{-d/2}e^{i(2\pi)^{-1}<\gamma^*,x>}
\eeq
for which we can write
\beq\label{dc2}
 \mathring{f}_\gamma =\underset{\gamma^*\in\Gamma_*}{\sum}\,\widehat{[f_\gamma]}_{\gamma^*}\,{\vartheta_{\gamma^*}}\,,
\eeq
with convergence of the series in $L^2$.

We are now ready to introduce our "magnetic" Gabor frame (for a general introduction to Gabor frames, see \cite{Chr}). We use an extra uni-modular factor containing a 'local gauge', which reminds on {the} choice proposed by Luttinger \cite{Lu}.
Given $A$ as in \eqref{dc1'}, we shall consider the following family of unitary operators on $L^2(\X)$, indexed by $\gamma\in\Gamma$:
$$
\big(\mathbf{\Lambda}^A_{\gamma}f\big)(x)\,:= \Lambda^A(x,\gamma)f(x),\quad\forall f\in L^2(\X).
$$
where { we use the shortcut notation:}
\beq\label{DF-LambdaA}
\Lambda^A(x,\gamma):=\,e^{-i\int_{[x,\gamma]}A}\,.
\eeq
{ We warn the reader that in \cite{CHP-3} we used the notation $e^{i\varphi(x,\gamma)}$ instead, with $\varphi(x,\gamma)=\int_{[\gamma,x]}A$; see \cite[Eq. (1.4)]{CHP-3}.}

\begin{definition}\label{D-G-frame}
Given $\Gamma\subset\X$ and $\mathcal{g}\in C^\infty_0(\X;[0,1])$ as in \eqref{dc3}, the family of functions:
\beq\label{DF-G-frame}
\mathcal{G}^A_{\gamma,\gamma^*}\,:=\,\mathbf{\Lambda}^A_\gamma\, \tau_{-\gamma}(\vartheta_{\gamma^*}\,\mathcal{g})\in\mathscr{S}(\X)
\eeq
is called the magnetic Gabor frame on $L^2(\X)$ associated with $(\Gamma, \mathcal{g})$.
\end{definition}
 We shall use the notation $\tGamma:=\Gamma\times\Gamma_*$, with elements of the form $\talpha:=(\alpha,\alpha^*),\tbeta:=(\beta,\beta^*),\ldots$. { We again warn the reader about two other changes of notation compared to \cite{CHP-3}. There, we directly identified an element of the form $(2\pi)^{-1}\gamma^*$ with some $m\in \Z^d$, thus the Gabor frame elements introduced in \cite{CHP-3} are indexed by two copies of $\Z^d$ instead of $\Gamma$ and $\Gamma_*$. Also, in that paper, the translation $\tau_\gamma$ acts like $(\tau_\gamma f)(x)=f(x-\gamma)$.  }

\subsection{Properties of the magnetic Gabor frame.}
 In this subsection we extend a technical result of \cite{CHP-3} in order to cover the $L^2$ case.

\begin{proposition}\label{P-Parseval-frame} We have the following properties: 

	\noindent {\rm (i)}.  The Gabor frame $\Big\{\mathcal{G}^A_{\tgamma}
	\Big\}$
	indexed by $\tgamma:=(\gamma,\gamma^*)\in\tGamma$ is a Parseval frame in $L^2(\X)$, i.e. the map: $$\mathfrak{U}^A_{\mathcal{g},\Gamma}:L^2(\X)\ni f\mapsto\big\{\big(\mathcal{G}^A_{\talpha}\,,\,f\big)_{L^2(\X)}\big\}\in\ell^2\big(\tGamma\big)$$ is an isometry.

	\noindent {\rm (ii)}. Given the magnetic Gabor frame $\Big\{\mathcal{G}^A_{\tgamma}
	\Big\}$,  we have for any $f\in L^2(\X)$ the identity:
	$$
	f\,=\,\underset{\talpha\in\tGamma}{\sum}\,\big(\mathcal{G}^A_{\talpha}\,,\,f\big)_{L^2(\X)}\,\mathcal{G}^A_{\talpha}
	$$
	with the above series converging in the $L^2(\X)$ norm.
\end{proposition}

\begin{proof}
	Instead of (i), we shall prove the following apparently stronger result:
	\beq\label{F-PF}
	\forall v,f\in L^2(\X),\quad\big(v,f\big)_{L^2(\X)}=\underset{\talpha\in\tGamma}{\sum}\,\big(v\,,\,\mathcal{G}^A_{\talpha}\big)_{L^2(\X)}\,\big(\mathcal{G}^A_{\talpha}\,,\,f\big)_{L^2(\X)}
	\eeq
with the series converging for the $\ell^2(\tGamma)$ norm. Let us compute the finite sums:
\begin{align*}
	&\underset{|\alpha|\leq N}{\sum}\,\underset{|\alpha^*|\leq M}{\sum}\,\big(v\,,\,\mathcal{G}^A_{\alpha,\alpha^*}\big)_{L^2(\X)}\,\big(\mathcal{G}^A_{\alpha,\alpha^*}\,,\,f\big)_{L^2(\X)}\\
	&\hspace*{1cm}=(2\pi)^{-d}\int_{\X}dx\int_{\X}dy\,\overline{v(x)}\,\,f(y)\underset{|\alpha|\leq N}{\sum}\,\Big(\Lambda^A(x,\alpha)\overline{\Lambda^A(y,\alpha)}\Big)\mathcal{g}(x-\alpha)\mathcal{g}(y-\alpha) \times \\
	&\hspace*{5cm}\times \underset{|\alpha^*|\leq M}{\sum}\,e^{(i/(2\pi))<\alpha^*,x-y>}.
\end{align*}
Define 
$$F_\alpha:= \mathcal{g}\, \tau_\alpha \Big (\overline{\Lambda^A(\cdot ,\alpha)} f\Big ),\quad V_\alpha:= \mathcal{g}\, \tau_\alpha \Big (\overline{\Lambda^A(\cdot ,\alpha)} v\Big ).$$
Then the above double series reads as 
\begin{align*}
&\underset{|\alpha|\leq N}{\sum}\,\underset{|\alpha^*|\leq M}{\sum}\,
\overline{\widehat{[V_\alpha]}_{\alpha^*}}\, \widehat{[F_\alpha]}_{\alpha^*}.
\end{align*}
The hypothesis concerning the functions $v,f,\mathcal{g},{\Lambda}^A(\cdot,\alpha)$ and the Parseval identity related to \eqref{dc2}  imply that the limit of the integral for $M\nearrow\infty$ exists and is equal to:
\begin{align*}
	\underset{M\nearrow\infty}{\lim}\,\underset{|\alpha|\leq N}{\sum}\,\underset{|\alpha^*|\leq M}{\sum}\,\big(v\,,\,\mathcal{G}^A_{\alpha,\alpha^*}\big)_{L^2(\X)}\,\big(\mathcal{G}^A_{\alpha,\alpha^*}\,,\,f\big)_{L^2(\X)}&=\underset{|\alpha|\leq N}{\sum}\,\int_{\X}dx\,\overline{v(x+\alpha)}\,f(x+\alpha )\,\mathcal{g}(x)^2\\ &=\int_{\X}dx\,\overline{v(x)}\,f(x)\,\underset{|\alpha|\leq N}{\sum}\,\mathcal{g}(x-\alpha)^2.
\end{align*}
The quadratic $\Gamma$-partition of unity property \eqref{dc3} of $\mathcal{g}$ together with the Dominated Convergence Theorem imply \eqref{F-PF} and thus (i).

Now let us prove (ii). From \eqref{F-PF} we know that the sequence
	$$
	f_{N,M}\,:=\,\underset{|\alpha|\leq N}{\sum}\,\underset{|\alpha^*|\leq M}{\sum}\,\big(\mathcal{G}^A_{\alpha,\alpha^*}\,,\,f\big)_{L^2(\X)}\,\mathcal{G}^A_{\alpha,\alpha^*}\,\in\,L^2(\X)
	$$
	converges to $f\in L^2(\X)$ in the weak topology on $L^2(\X)$. To prove its convergence for the norm topology on $L^2(\X)$ we show that it is a Cauchy sequence. Let us choose $N<N'$ and $M<M'$ and compute:
	\begin{align*}
		&\big\|f_{N',M'}-f_{N,M}\big\|_{L^2(\X)}=\underset{\|v\|_{L^2(\X)}=1}{\sup}\,\big|\big(v\,,\, f_{N',M'}-f_{N,M}\big)_{L^2(\X)}\big|\\
  &=\underset{\|v\|_{L^2(\X)}=1}{\sup}\,\Big|\underset{|\alpha|\leq N'}{\sum}\,\underset{|\alpha^*|\leq M'}{\sum}\,\big(v\,,\,\mathcal{G}^A_{\alpha,\alpha^*}\big)_{L^2(\X)}\,\big(\mathcal{G}^A_{\alpha,\alpha^*}\,,\,f\big)_{L^2(\X)}- \\
  &\qquad - \underset{|\alpha|\leq N}{\sum}\,\underset{|\alpha^*|\leq M}{\sum}\,\big(v\,,\,\mathcal{G}^A_{\alpha,\alpha^*}\big)_{L^2(\X)}\,\big(\mathcal{G}^A_{\alpha,\alpha^*}\,,\,f\big)_{L^2(\X)}\Big|\\
		&=\underset{\|v\|_{L^2(\X)}=1}{\sup}\,\Big|\Big (\underset{N<|\alpha|\leq N'}{\sum}\,\underset{|\alpha^*|\leq M'}{\sum}\, +\underset{|\alpha|\leq N}{\sum}\,\underset{M<|\alpha^*|\leq M'}{\sum}\Big )\,\big(v\,,\,\mathcal{G}^A_{\alpha,\alpha^*}\big)_{L^2(\X)}\,\big(\mathcal{G}^A_{\alpha,\alpha^*}\,,\,f\big)_{L^2(\X)}\Big|\\
		&\leq\left(\Big (\underset{N<|\alpha|}{\sum}\,\underset{\alpha^*\in \Gamma_*}{\sum}  +\underset{\alpha\in \Gamma }{\sum}\,\underset{M<|\alpha^*|}{\sum}\,\Big )\,\big|\big(\mathcal{G}^A_{\alpha,\alpha^*}\,,\,f\big)_{L^2(\X)}\big|^2\right)^{1/2},
	\end{align*}
 where in the last inequality we used the Cauchy-Schwarz inequality and (i) in order to get rid of $v$.
	Using once again (i) we may conclude that the above remainder converges to $0$ for $N$ and $M$ going to $\infty$, thus $f_{N,M}$ converges in norm, and its strong limit must equal $f$.
\end{proof}

\subsection{Infinite matrices associated with  operators in a magnetic Gabor frame.}

Suppose that we have  a continuous operator $T:\mathscr{S}(\X)\rightarrow \mathscr{S}^\prime(\X)$. Notice that any  $T\in\mathcal L \big(L^2(\X)\big)$ and  the closure of any symmetric operator on $\mathscr{S}(\X)$ { are examples of such operators}. Given {the magnetic Gabor frame} \eqref{DF-G-frame} we can associate with   it the following \say{infinite matrix}:
\beq\label{dc4}
\mathbb{M}^A[T]_{\talpha,\tbeta}\,:=\,\big\langle T\mathcal{G}^A_{\tbeta}\,,\,\overline{\mathcal{G}^A_{\talpha}}\big\rangle_{{\mathscr{S}', \mathscr{S}}}.
\eeq
Frequently we shall write the above matrix elements using Formula \eqref{F-dual-pscal}.

\begin{notation}
	We shall denote by $\mathscr{M}_\Lambda$ the complex linear space of infinite matrices with complex entries indexed by a regular lattice $\Lambda$.  We shall work with the lattices $\Gamma, \Gamma_*$ and $\tGamma=\Gamma\times\Gamma_*$.
\end{notation}
 Given any linear operator $T:\mathscr{S}(\X)\rightarrow L^2(\X)$, one may consider it as an operator in the Hilbert space $L^2(\X)$ with domain $\mathscr {S}(\X)$ and define its adjoint $T^*:\mathcal{D}(T^*)\rightarrow L^2(\X)$ putting:$$\mathcal{D}(T^*):=\big\{v\in L^2(\X)\,,\,\big|\big(v,T\varphi\big)_{L^2(\X)}\big|\leq C(T,v)\|\varphi\|_{L^2(\X)}\,,\, \forall\varphi\in\mathscr{S}(\X)\big\}$$ and $\big(T^*v,\varphi\big)_{L^2(\X)}:=\big(v,T\varphi\big)_{L^2(\X)}$ for any $\varphi\in\mathscr{S}(\X)$.

\begin{proposition}\label{P-GF-Op-dec}
	{Let } $T:\mathscr{S}(\X)\rightarrow L^2(\X)$ be as above and assume that  $\mathscr{S}(\X)\subset\mathcal{D}(T^*)$. Then  for any $\varphi\in \mathscr{S}(\X)$, we have   the identity in $L^2(\X)$:
	\beq\label{dc5}
	T\varphi\,=\,\underset{\talpha\in\tGamma}{\sum}\,\left(\underset{\tbeta\in\tGamma}{\sum}\,\mathbb{M}^A[T]_{\talpha,\tbeta}\Big(\mathcal{G}^A_{\tbeta}\,,\,\varphi\Big)_{L^2(\X)}\right)\,\mathcal{G}^A_{\talpha}\,,
	\eeq
	{ where} the series  indexed by $\tGamma$ converge in the norms of  $\ell^2(\tGamma)$ and $L^2(\X)$ resp.\,. { Also, if $\psi, \varphi\in \mathscr{S}(\X)$ then
 \beq\label{dc5'}
	\big ( \psi,T\varphi\big )_{L^2(\X)}\,=\,\underset{\talpha\in\tGamma}{\sum}\,\Big(\psi\, ,\, \mathcal{G}^A_{\talpha}\Big)_{L^2(\X)}\left(\underset{\tbeta\in\tGamma}{\sum}\,\mathbb{M}^A[T]_{\talpha,\tbeta}\Big(\mathcal{G}^A_{\tbeta}\,,\,\varphi\Big)_{L^2(\X)}\right)\,.
	\eeq
 
 }
\end{proposition}
\begin{proof}
	We use Proposition \ref{P-Parseval-frame} for $T\varphi\in L^2(\X)$ and write:
	\begin{align*}
		T\varphi\,&=\,\underset{\talpha\in\tGamma}{\sum}\,\Big(\mathcal{G}^A_{\talpha}\,,\,T\varphi\Big)_{L^2(\X)}\,\mathcal{G}^A_{\talpha}\,=\,\underset{\talpha\in\tGamma}{\sum}\,\Big(T^*\mathcal{G}^A_{\talpha}\,,\,\varphi\Big)_{L^2(\X)}\,\mathcal{G}^A_{\talpha}
	\end{align*}
with the series converging in $L^2(\X)$. Using once again  Proposition \ref{P-Parseval-frame}  for $\varphi\in\mathscr{S}(\X)$ we obtain that
\begin{align*}
\Big(T^*\mathcal{G}^A_{\talpha},\varphi\Big)_{L^2(\X)}=\underset{\tbeta\in\tGamma}{\sum}\,\big(T^*\mathcal{G}^A_{\talpha},\big(\mathcal{G}^A_{\tbeta},\varphi\big)_{L^2(\X)}\mathcal{G}^A_{\tbeta}\big)_{L^2(\X)}=\underset{\tbeta\in\tGamma}{\sum}\,\big(\mathcal{G}^A_{\talpha},T\mathcal{G}^A_{\tbeta}\big)_{L^2(\X)}\,\big(\mathcal{G}^A_{\tbeta},\varphi\big)_{L^2(\X)}
\end{align*}
with the series converging in the norm of  $\ell^2(\tGamma)$. { This proves \eqref{dc5}. Concerning \eqref{dc5'}, we use again Proposition \ref{P-Parseval-frame} and write 

\begin{align*}
	\big ( \psi,T\varphi\big )_{L^2(\X)}\,&=\,
 \,\underset{\talpha\in\tGamma}{\sum}\,\Big(\psi\, ,\, \mathcal{G}^A_{\talpha}\Big)_{L^2(\X)}\big(T^*\mathcal{G}^A_{\talpha}\, ,\,\varphi\big)_{L^2(\X)}\\
 &=\, 
 \underset{\talpha\in\tGamma}{\sum}\,\Big(\psi\, ,\, \mathcal{G}^A_{\talpha}\Big)_{L^2(\X)}\,
 \underset{\tbeta\in\tGamma}{\sum}\,\mathbb{M}^A[T]_{\talpha,\tbeta}\Big(\mathcal{G}^A_{\tbeta}\,,\,\varphi\Big)_{L^2(\X)}\,.
	\end{align*}

}
\end{proof}

\section{The matrix form of the magnetic Weyl calculus in a magnetic Gabor frame.} 

\subsection{Brief reminder of the magnetic Weyl calculus.}

Let us recall the magnetic Weyl quantization \cite{MP-1,IMP-1,IMP-19}. Given $\Phi \in \mathscr{S}(\Xi)$, we define its quantization  $ \Op^A(\Phi)\in\mathcal{L}\big(\mathscr{S}(\X);\mathscr{S}(\X)\big),$ by
\begin{equation*}
\big(\Op^A(\Phi)\varphi\big)(x):={ (2\pi)^{-d}}\int_{\X^*}d\xi\int_{\X}dy\,\Lambda^A(x,y)\,\,e^{i<\xi,x-y>}\,\Phi\big((x+y)/2,\xi\big)\,\varphi(y), \forall\varphi\in\mathscr{S}(\X)
\end{equation*}
where we used the notation introduced in \eqref{DF-LambdaA}.
 It has been proven in \cite{MP-1} (Proposition 3.5) that this \say{quantization} $\Phi \mapsto \Op^A(\Phi)$ may be extended to the following isomorphism of topological linear spaces $\Op^A$ from $\mathscr{S}^\prime(\Xi)$ onto $\mathcal{L}\big(\mathscr{S}(\X);\mathscr{S}^\prime(\X)\big)$  defined by 
\begin{align*}
\langle\Op^A(\Phi)\phi,\psi\rangle_{\mathscr S',\mathscr S}={ (2\pi)^{-d}}{ \int_{\X^*}d\xi\int_{\X}dx}\int_{\X}dy\,\Lambda^A(x,y)\,\,e^{i<\xi,x-y>}\Phi\big(\frac{x+y}{2},\xi\big)\phi(y)\psi(x),\\
\forall(\phi,\psi)\in\mathscr{S}(\X)\times\mathscr{S}(\X).
\end{align*}

We shall call $F\in\mathscr{S}^\prime(\Xi)$ \textit{the distribution symbol of} $\Op^A(F)\in\mathcal{L}\big(\mathscr{S}(\X);\mathscr{S}^\prime(\X)\big)$. \\
Later on we shall work more particularly  with the H\"{o}rmander classes of symbols indexed by $p\in\mathbb{R}$:
\beq\begin{split}\label{D-Hsymb}
S_0^p(\X\times\X^*)\,:&=\,\big\{F\in C^\infty(\Xi)\mbox{ s.t. } \nu^{p,\rho}_{n,m}(F)<\infty\,,\,\forall(n,m)\in\mathbb{N}\times\mathbb{N}\big\},\\
\text{where:}&\ \nu^{p}_{n,m}(F):=\underset{|\alpha|\leq n}{\max}\,\underset{|\beta|\leq m}{\max}\,\underset{(x,\xi)\in\Xi}{\sup}<\xi>^{-p}\big|\big(\partial_x^\alpha\partial_\xi^\beta F\big)(x,\xi)\big|.
\end{split}\eeq

\subsection{ The main results}

Let us come back to our Gabor frame $\big\{\mathcal{G}^{A}_{\gamma,\gamma^*}\big\}_{(\gamma,\gamma^*)\in\Gamma\times\Gamma_*}$ and compute the associated matrix for an operator of the form $\Op^A(\Phi)$ for some $\Phi\in\mathscr{S}^\prime(\Xi)$  (using \eqref{F-dual-pscal}):
\begin{align}\label{F-MOpAPhi}
&\mathbb{M}^A[\Op^A(\Phi)]_{\talpha,\tbeta}:=\Big(\mathcal{G}^A_{\talpha}\,,\,\Op^A(\Phi)\mathcal{G}^A_{\tbeta}\Big)_{L^2(\X)}={(2\pi)^{-2d}}\int_{\X}dx\int_{\X}dy\int_{\X^*}d\eta\, \Lambda^A(\alpha,x)\,\times\\\nonumber
&\times\,e^{-\frac{i}{2\pi} <\alpha^*,x-\alpha>}\mathcal{g}(x-\alpha)\,\Lambda^A(x,y)\,e^{i<\eta,x-y>}\,\Phi\big((x+y)/2,\eta\big)\,\Lambda^A(y,\beta)\,e^{\frac{i}{2\pi}<\beta^*,y-\beta>}\mathcal{g}(y-\beta).
\end{align}
Let us denote by $<x,y,z>:=\big\{u:=x+t(y-x)+st(z-y)\in\X\,,\,(t,s)\in[0,1]\times[0,1]\big\}$ the triangle with vertices $\{x,y,z\}$ and by
$$
\Omega^B(x,y,z)\,:=\,e^{-i\int_{<x,y,z>}B}.
$$
Then we notice that Stokes' formula implies the identity:
\beq\label{F-Stokes}
\Lambda^A(\alpha,x)\Lambda^A(x,y)\Lambda^A(y,\beta)=\Lambda^A(\alpha,\beta)\Omega^B(\alpha,x,y)\Omega^B(\alpha,y,\beta)
\eeq
and for our class of magnetic fields introduced in Subsection \ref{ss1.3} we also have the following estimates:
\beq\begin{split}\label{E-OmegaB}
{ \big|\partial_x^a\partial_y^b\partial_z^c\Omega^B(x,y,z)\big|\,\leq\,C_{ab}(B)\, \big (1+\text{diameter}(<x,y,z>)+\text{area}(<x,y,z>)\big )^{|a+b+c|}.}
\end{split}
\eeq

Here is the main  result of the paper.

\begin{theorem}\label{T-m-Gabor-frame}   Given some $p\in \mathbb R$ and some   
 $\Phi \in  \mathscr{S}^\prime(\Xi)$, the following two statements
 are equivalent: 

\noindent {\rm (i)}	$\Phi$  belongs to $S^p_0(\X\times \X^*)$.

\noindent {\rm (ii)} For any $(n_1,n_2)\in\mathbb{N}^2$ there exists some constant $C_{n_1,n_2}(\Phi,B)>0$ such that the $\tGamma$-indexed matrix of $\Op^A(\Phi)$ in the magnetic Gabor frame $\big\{\mathcal{G}^A_{\tgamma}\big\}_{\tgamma\in\tGamma}$ has the following behavior:
	\beq\label{F-Mdecay}
	\underset{(\talpha,\tbeta)\in\tGamma^2}{\sup}<\alpha-\beta>^{n_1}<\alpha^*-\beta^*>^{n_2}<\alpha^*+\beta^*>^{-p}\,\big|\mathbb{M}^A[\Op^A(\Phi)]_{\talpha,\tbeta}\big|\,\leq\,C_{n_1,n_2}(\Phi,B).
	\eeq
\end{theorem}
\noindent {\it Proof that (i) implies (ii)}.\\
  Let us assume that (i) holds. Let us start from \eqref{F-MOpAPhi} and make the following change of variables:
	$$\X\times\X\ni (x,y)\mapsto (z,v):=\big((x+y-\alpha-\beta)/2,x-y-\alpha+\beta\big)\in\X\times\X$$
	and also the following similar bijective change of indices for the $\tGamma\times\tGamma$-indexed series:
	\beq\begin{array}{lcl}\label{F-chvar-2-b}
		\Gamma\times\Gamma\ni(\alpha,\beta)&\mapsto&(\mu,\nu):=\big(\alpha+\beta,\alpha-\beta\big)\in[\Gamma]^2_\circ,\\
		\Gamma_*\times\Gamma_*\ni(\alpha^*,\beta^*)&\mapsto&(\mu^*,\nu^*):=\big(\alpha^*+\beta^*,\beta^*-\alpha^*\big)\in[\Gamma^*]^2_\circ,
	\end{array}\eeq
	where $$[\Gamma]^2_\circ:=\{(\mu,\nu)\in\Gamma\times\Gamma\,,\,\text{\it  where $\mu_j$ and $\nu_j$ for $1\leq j\leq d$ are simultaneously even or odd}\}$$ and similarly for $\Gamma_*$. With these new indices we consider the corresponding matrix $$\widetilde{\mathbb{M}}[\Op^A(\Phi)]_{\tmu(\talpha,\tbeta),\tnu(\talpha,\tbeta)}:=\mathbb{M}[\Op^A(\Phi)]_{\talpha,\tbeta}.$$ Using \eqref{F-Stokes} and {  making the change $\zeta=\eta-\mu^*/(4\pi)$} we obtain that:
	\begin{align*} \numberthis\label{F-tM-OpAPhi}
		&\widetilde{\mathbb{M}}[\Op^A(\Phi)]_{\tmu,\tnu}\\
   &\quad ={ (2\pi)^{-2d}}\Lambda^A\big((\mu+\nu)/2,(\mu-\nu)/2\big)e^{\frac{i}{4\pi}<\mu^*,\nu>}\int_{\X}dz\,e^{\frac{i}{2\pi}<\nu^*,z>}\,\int_{\X^*}d\zeta\,e^{i<\zeta,\nu>}\,\times\\
		&\qquad\times\int_{\X}dv\,e^{i<\zeta,v>}\,\mathcal{g}(z+v/2)\mathcal{g}(z-v/2)\,\Theta^B_{\mu,\nu}(z,v)\,\Phi\big(z+{ \mu/2},\zeta+\mu^*/(4\pi)\big),
	\end{align*}
where
$$
\Theta^B_{\mu,\nu}(z,v):=\Omega^B(\alpha(\mu,\nu),x(z,v),y(z,v))\,\Omega^B(\alpha(\mu,\nu),y(z,v),\beta(\mu,\nu)).
$$

Let us fix some arbitrary $(n_1,n_2)\in\mathbb{N}^2$ and rewrite \eqref{F-tM-OpAPhi} (for any $(m_1,m_2,m_3)\in\mathbb{N}^3$):
	\begin{align*}\numberthis\label{F-est-tM-OpAPhi}
		&\Big|\widetilde{\mathbb{M}}^A[\Op^A(\Phi)]_{\tmu,\tnu}\Big|\leq (2\pi)^{-2d}\times \\
		&  \Big|\int_{\X^*}d\zeta\,\int_{\X\times \X}dz\,dv\,\Big[\Big(\frac{1-\Delta_z}{<\nu^*>^2}\Big)^{m_1}\hspace*{-8pt}e^{\frac{i}{2\pi}<\nu^*,z>}\Big]\Big[\Big(\frac{1-\Delta_\zeta}{<\nu>^2}\Big )^{m_2}e^{i<\zeta,\nu>}\Big]\,\Big [\Big(\frac{1-\Delta_v}{<\zeta>^2}\Big )^{m_3} e^{i<\zeta,v>}\Big]\\
		&\qquad\qquad \qquad \qquad \qquad\qquad  \times\,\mathcal{g}(z+v/2)\mathcal{g}(z-v/2)\,\Theta^B_{\mu,\nu}(z,v)\,\Phi\big(z+{ \mu/2},\zeta+{ \mu^*/(4\pi)}\big)\Big|.
	\end{align*}
	Since the support of $\mathcal{g}$ is included in $(-1,1)^d$, we may assume $|z_j\pm v_j/2|\leq 1$ for all $1\leq j\leq d$. Thus 
	$$2|z_j|= |z_j-v_j/2+z_j+v_j/2|\leq 2,\quad |v_j|= |z_j+v_j/2-(z_j-v_j/2)|\leq 2,$$
	 and  the integral with respect to $z$ and $v$ in \eqref{F-est-tM-OpAPhi}  is restricted to $|z_j|\leq 1$ and $|v_j|\leq 2$ for all $1\leq j\leq d$. Using the bounds in \eqref{E-OmegaB} and restricting to $|z_j|\leq 1,|v_j|\leq 2$ implies:
	\beq\label{dc6}\begin{split}
	\big|\partial_z^a\partial_v^b\Theta^B_{\mu,\nu}(z,v)\big|\,&\leq\,C_{ab}(B)<z+v/2>^{|a+b|}<v>^{|a+b|}<\nu>^{|a+b|}<z-v/2>^{|a+b|}\\
	&\leq\,C'_{a,b}(B)<\nu>^{|a+b|}.
\end{split}\eeq
 
Integrating by parts with respect to $z$ we obtain a decay in $\nu^*$ of the type $<\nu^*>^{-2m_1}$, at the price of up to $2m_1$ derivatives acting on $\Theta^B_{\mu,\nu}(z,v)$, which produce a growth like $<\nu>^{2m_1}$. Then integrating by parts with respect to $v$ we produce a decaying factor $<\zeta>^{-2m_3}$ at a price of up to $2m_3$ derivatives acting on $\Theta^B_{\mu,\nu}(z,v)$, which produce another growth like $<\nu>^{2m_3}$. Finally, integrating by parts with respect to $\zeta$ will produce a decay $<\nu>^{-2m_2}$ together with some powers of $v$ (they are bounded on the domain of integration), while the decay $<\zeta>^{-2m_3}$ is not affected (it can only be improved). Hence \eqref{F-est-tM-OpAPhi} reads as:
\begin{align*}
   \Big|\tilde{\mathbb{M}}^A[\Op^A(\Phi)]_{\tmu,\tnu}\Big|\leq C <\nu^*>^{-2m_1}<\nu>^{2m_1+2m_3-2m_2} \int_{\X^*} d\zeta  <\zeta>^{-2m_3}\,  { <\zeta+\mu^*/(4\pi)>^p}.
\end{align*}
Now we choose $2m_1\geq n_1$, $2m_3>|p|+d+1$ and $2m_2>2m_1+2m_3+n_2$. The decay in $\nu^*$ and $\nu$ holds,  the function $<\zeta>^{-d-1}$ is integrable, hence we only need to check that the quantity
\begin{equation}\label{dhc5}
<\mu^*/(4\pi)>^{-p}\, <\zeta>^{-|p|}\, { <\zeta+\mu^*/(4\pi)>^p}
\end{equation}
is uniformly bounded in $\zeta$ and $\mu^*$ for any $p\in\R$. 

{ Let us show that for every $s\in \R$ we have 
\begin{equation}\label{dhc2}
\begin{split}
&<x+y>^s\, \leq \, {2^{|s|/2}}<x>^s\; <y>^{|s|}, \quad \text{or equivalently}\\
&<x+y>^s\, <x>^{-s}\; <y>^{-|s|}\leq {2^{|s|/2}}.
\end{split}
\end{equation}
When $s>0$ we can reduce it to $1+|x+y|^2\leq 2(1+|x|^2)(1+|y|^2)$, while when $s<0$ we use that $s=-|s|$ and 
$$<x>^{|s|}=<(x+y)+(-y)>^{|s|}\leq 2^{|s|/2}<x+y>^{|s|}\, <y>^{|s|}.$$
Now use $s=p$, $x=-\mu^*/(4\pi)$ and $y=\zeta +\mu^*/(4\pi)$ in \eqref{dhc2} and we get that the quantity in \eqref{dhc5} is bounded by $2^{|p|/2}$. 
}
This ends the proof of (ii). 

\vspace{0.5cm}

\noindent {\it Proof that (ii) implies (i)}.\\ Suppose that $\Phi\in\mathscr{S}^\prime(\Xi)$ is such that \eqref{F-Mdecay} is valid.
	Then $\Op^A(\Phi)\in\mathcal{L}\big(\mathscr{S}(\X);\mathscr{S}^\prime(\X)\big)$ and has a distribution kernel given by:
	$$
	\mathfrak{K}^A[\Phi]\,=\,(2\pi)^{-d/2}\Lambda^A\,\big[\big(\Upsilon^*\circ(\bb1_{\X}\otimes\mathcal{F}_{\X^*}^*)\big)\Phi\big]
	$$
	where $\Upsilon^*:\mathscr{S}^\prime(\X\times\X)\rightarrow\mathscr{S}^\prime(\X\times\X)$ is the extension to tempered distributions of the change of variables map $\X\times\X\ni(x,y)\mapsto\big((x+y)/2,(x-y)\big)\in\X\times\X$. Using the magnetic Gabor frame and Proposition \ref{P-GF-Op-dec} we can write:
	$$
	\mathfrak{K}^A[\Phi]\,=\,\underset{(\talpha,\tbeta)\in\tGamma\times\tGamma}{\sum}\mathbb{M}^A[\Op^A(\Phi)]_{\talpha,\tbeta}\Big(\mathcal{G}^A_{\talpha}\otimes\overline{\mathcal{G}^A_{\tbeta}}\Big)
	$$
	where each term belongs to $\mathscr{S}(\X\times\X)$ and the series converges in the weak { distributional} sense. Thus, we shall approximate the distribution kernel $\mathfrak{K}^A[\Phi]$, in the weak tempered distribution topology,   by a family $\mathfrak{K}^A_N[\Phi]$  indexed by $N\in \mathbb N$ of integral kernels of class $\mathscr{S}(\X\times\X)$ defined by:
	$$
	\mathfrak{K}^A_N[\Phi]\,:=\,\underset{|\talpha|\leq N,|\tbeta|\leq N}{\sum}\mathbb{M}^A[\Op^A(\Phi)]_{\talpha,\tbeta}\Big(\mathcal{G}^A_{\talpha}\otimes\overline{\mathcal{G}^A_{\tbeta}}\Big)\,.
	$$
 The symbols associated to the distribution kernels $\mathfrak{K}^A_N[\Phi]$,
 denoted by $\Phi_N\in\mathscr{S}(\Xi)$, { are} defined by:
	\begin{align*}
		\Phi_N(z,\zeta):&=\int_{\X}dv\,e^{-i<\zeta,v>}\,{ \Lambda^A\big(z-v/2,z+v/2\big) }\,\mathfrak{K}^A_N[\Phi]\big(z+v/2,z-v/2\big)\\ \nonumber
		&=\,\underset{|\talpha|\leq N,|\tbeta|\leq N}{\sum}\mathbb{M}^A[\Op^A(\Phi)]_{\talpha,\tbeta}\,\times \\ \nonumber
		&\quad \times\,\int_{\X}dv\,e^{-i<\zeta,v>}\,\Lambda^A\big(z-v/2,z+v/2\big)\,\mathcal{G}^A_{\talpha}\big(z+v/2\big)\, \overline{\mathcal{G}^A_{\tbeta}\big(z-v/2\big)}\,.
	\end{align*}
	Let us compute:
	\begin{align}\label{dc7}
		&\int_{\X}dv\,e^{-i<\zeta,v>}\,\Lambda^A\big(z-v/2,z+v/2\big)\,\mathcal{G}^A_{\talpha}\big(z+v/2\big)\,\overline{\mathcal{G}^A_{\tbeta}\big(z-v/2\big)}\\
		& =(2\pi)^{-d}\int_{\X}dv\,e^{-i<\zeta,v>}\,\Lambda^A\big(z-v/2,z+{ v/2}\big)\Lambda^A\big(z+v/2,\alpha\big)\Lambda^A\big(\beta,z-v/2\big)\,\times \nonumber \\
		&\quad\times\,e^{\frac{i}{2\pi}<\alpha^*,z+v/2-\alpha>}\,e^{-\frac{i}{2\pi}<\beta^*,z-v/2-\beta>}\,\g\big(z+v/2-\alpha\big)\,\g\big(z-v/2-\beta\big).\nonumber
	\end{align}
	We make the change of variables:
	$$
	\left\{\begin{array}{l}
		\X\ni z\mapsto z':=z-(\alpha+\beta)/2\in\X,\\
		\X\ni v\mapsto v':=v-(\alpha-\beta)\in\X,
	\end{array}\right.
	$$
	and 
	$$
	\left\{\begin{array}{l}
		\Gamma\times\Gamma\ni (\alpha,\beta)\mapsto(\mu,\nu):=\big(\alpha+\beta,\alpha-\beta\big)\in\Gamma\times\Gamma,\\
		\Gamma_*\times\Gamma_*\ni
		(\alpha^*,\beta^*)\mapsto(\mu^*,\nu^*):=\big(\alpha^*+\beta^*,\alpha^*-\beta^*\big)\in\Gamma_*\times\Gamma_*,
	\end{array}\right.
	$$
	and introduce
	\begin{align*}
\Theta^B(z',v',\mu,\nu)&:=\Lambda^A(\beta,\alpha)\Lambda^A\big(z-v/2,z+v/2\big)\Lambda^A\big(z+v/2,\alpha\big)\Lambda^A\big(\beta,z-v/2\big)\\
&=\Omega^B\big(z'-v'/2+\beta,z'+v'/2+\alpha,\alpha\big)\,\Omega^B\big(z'-v'/2+\beta,\alpha,\beta\big).
	\end{align*}
Then the integral in \eqref{dc7} reads as:
	\begin{align}\label{dc8}
		&(2\pi)^{-d}\Lambda^A(\alpha,\beta)\int_{\X}dv'\,e^{-i<\zeta,v'+\nu>}\,e^{\frac{i}{2\pi}<\nu^*,z'>}\,e^{\frac{i}{4\pi}<\mu^*,v'>}\times \nonumber \\
		&\hspace*{4cm}\times\,\Theta^B(z',v',\mu,\nu)\, \mathcal{g}(z'+v'/2)\, \mathcal{g}(z'-v'/2).
	\end{align}
	On the support of $\mathcal{g}\in C^\infty_0(\X)$ we have $z'\pm v'/2\in (-1,1)^d$. Thus we must have $|v'_j|<2$ and $|z'_j|<1$, where the second inequality implies $$2z_j-2<\mu_j<2z_j+2,\quad \forall \, 1\leq j\leq d.$$  This last condition { implies } that there exists some finite set $\Sigma(z)\subset\Gamma$ with $\#\Sigma(z)\leq 5^d$ such that for $\mu\notin\Sigma(z)$ the above integral vanishes.
	
	From the definition of $\Phi_N$ we see that we need to control multiple series involving  $\mu$, $\nu$, $\nu^*$ and $\mu^*$. We have just seen that there are only finitely many $\mu$'s which contribute, uniformly in $z$ and $\zeta$. The series in $\nu$ and $\nu^*$ will be controlled by using the strong decay of the matrix elements, so we only need to worry about the sum over $\mu^*$. In order to get some decay in $\mu^*$ we have to perform some partial integration. 
	
	Let us consider the \say{image} lattices 
	\[\begin{split}
			&\hat{\Gamma}^2:=\big\{(\mu,\nu)\in\Gamma\times\Gamma\,,\,\big((\mu+\nu)/2,(\mu-\nu)/2\big)\in\Gamma\times\Gamma\big\}\\
			&\hat{\Gamma}^2_N:=\big\{(\mu,\nu)\in\hat{\Gamma}^2\,,\,|(\mu+\nu)/2|\leq N,\,|(\mu-\nu)/2|\leq N\big\},
	\end{split}
	\]
	and similarly $[\widehat{\Gamma_*}]^2$ and $[\widehat{\Gamma_*}]^2_N$. Then we can write:
	\begin{align*}
		&\Phi_N(z,\zeta)=(2\pi)^{-d}
		\underset{\mu\in\Sigma(z),(\mu,\nu)\in\hat{\Gamma}^2_N}{\sum}\hspace*{-24pt}\Lambda^A\big((\mu+\nu)/2,(\mu-\nu)/2\big)\hspace*{-18pt} \underset{(\mu^*,\nu^*)\in[\widehat{\Gamma_*}]^2_N}{\sum}\hspace*{-12pt}\mathbb{M}^A[\Op^A(\Phi)]_{\tmu(\talpha,\tbeta),\tnu(\talpha,\tbeta)}\,\times  \\
		&\times \, \hspace*{-12pt}\int\limits_{\underset{1\leq j\leq d}{\prod}\{|v'_j|\leq2\}}\hspace*{-12pt}dv' e^{-i<\zeta-\mu^*/(4\pi),v'>}\, e^{-i<\zeta,\nu>}\,e^{\frac{i}{2\pi}<\nu^*,z'>}\,\g\big(z'+v'/2\big)\,\g\big(z'-v'/2\big)\,\Theta^B(z',v',\mu,\nu).
	\end{align*}
	
	A crucial observation is that, for any multi-indices $(a,b)\in\mathbb{N}^d\times\mathbb{N}^d$, if we consider $\big(\partial_z^a\partial_\zeta^b\Phi_N\big)(z,\zeta)$,  we can generate powers of $v'$, of $\nu$ and $\nu^*$. On the support of $\mathcal{g}$, the variables $v_j'$ are bounded. Integrating by parts $M\geq d+1+|p|$ times with respect to $v'$ we can make appear a factor of the type 
	$<\zeta -\mu^*/(4\pi)>^{-d-1-|p|},$
	at the price of some extra powers of $\nu$. Due to \eqref{F-Mdecay} we see that the summation over the indices $\nu$ and $\nu^*$ are under control, and we only need to bound the series 
	$$\sum_{\mu^*\in \Gamma_*}\, <\zeta -\mu^*/(4\pi)>^{-d-1}\, <\mu^*/(4\pi)>^p<\zeta -\mu^*/(4\pi)>^{-|p|}.$$

 Using \eqref{dhc2} with $x=-\zeta$ and $y=\zeta -\mu^*/(4\pi)$ we have 
 $$<\mu^*/(4\pi)>^p<\zeta -\mu^*/(4\pi)>^{-|p|}\leq 2^{|p|/2}\, <\zeta>^p,$$
 thus the series with respect to $\mu^*$ can also be bounded by a constant times $<\zeta>^p$, hence we have just proved that for any pair of multi-indices $a,b$ there exists a constant $C_{a,b}$ such that for any $N\geq 1$ we have 
	$$
	\big|\big(\partial_z^a\partial_\zeta^b\Phi_N\big)(z,\zeta)\big|\,\leq\,C_{a,b}\, <\zeta>^p,\quad\forall(z,\zeta)\in\Xi\,.
	$$
	Notice, that the above estimate is uniform for $z$ and $\zeta$ restricted to compact sets, and the symbol $\Phi$ will be the uniform limit of $\Phi_N$ on compact sets when $N\to +\infty$. 
	
\qed

\begin{notation}\label{N-rdec-offd-m}
For $p\in \mathbb R$, we  denote by $\mathscr{M}^p_{\tGamma,\infty}$ the complex linear space of  infinite matrices indexed by the lattice $\tGamma$ and verifying the estimate \eqref{F-Mdecay}. We say that they have \textit{rapid off-diagonal decay}.
\end{notation}

\begin{proposition}\label{P-comp-rdec-offd-m}
Let $p,q\in \mathbb R$, and let $(\M,\M^\prime)\in\mathscr{M}^p_{\tGamma,\infty}\times\mathscr{M}^q_{\tGamma,\infty}$. Then their matrix product is an element of $\mathscr{M}^{p+q}_{\tGamma,\infty}$.
\end{proposition}
\begin{proof}
{ We apply the inequality \eqref{dhc2} twice, first with $s=p$, $x=\alpha^*+\gamma^*$ and $y=-\gamma^*+\beta^*$, and second with $s=q$, $x=\beta^*+\gamma^*$ and $y=-\gamma^*+\alpha^*$, } and obtain  
$$<\alpha^* +\beta^*>^{p+q}\, \leq \, { 2^{(|p|+|q|)/2}}<\alpha^*+\gamma^*>^p <\beta^*-\gamma^*>^{|p|} <\beta^*+\gamma^*>^q <\alpha^*-\gamma^*>^{|q|}.   $$
Then for any $(m_1,m_2)\in\mathbb{N}\times\mathbb{N}$  we have the bound:
\begin{align*}
&<\alpha-\beta>^{m_1}\, <\alpha^*-\beta^*>^{m_2} \, <\alpha^*+\beta^*>^{p+q}\big|(\M\cdot\M^\prime)_{\talpha,\tbeta}\big|\\
&\qquad \leq {2^{(|p|+|q|+m_1+m_2)/2}}\underset{\tgamma\in\tGamma}{\sum}\Big (<\alpha-\gamma>^{m_1}\, <\alpha^*-\gamma^*>^{m_2+|q|} \, <\alpha^*+\gamma^*>^{p}|\M_{\talpha,\tgamma}| \Big ) \times \\
&\qquad \qquad \times \Big (\, <\gamma-\beta>^{m_1}\, <\gamma^*-\beta^*>^{m_2+|p|} \, <\gamma^*+\beta^*>^{q}|\M^\prime_{\tgamma,\tbeta}|\Big).
\end{align*}
Using \eqref{F-Mdecay} with $n_1>m_1+d$ and $n_2>m_2+|p|+|q| +d$, we see that the series on the right hand side converges and is uniform in $\alpha,\alpha^*,\beta,\beta^*$.
\end{proof}

\begin{proposition}\label{P-prodMAOp}
	Let $p,q\in\mathbb{R}$, let $(\Phi,\Psi)\in S^p_0(\X\times \X^*)\times S^q_0(\X\times \X^*)$,  and consider their infinite matrices with respect to a magnetic Gabor frame $\big\{\mathcal{G}^A_{\alpha,\alpha^*}\big\}_{(\alpha,\alpha^*)\in\Gamma\times\Gamma_*}$. Then:
	\begin{subequations}\label{F-prodMAOp}
	\beq	\mathbb{M}^A\big(\Op^A(\Phi)\Op^A(\Psi)\big)\,=\,\mathbb{M}^A\big(\Op^A(\Phi)\big)\cdot\mathbb{M}^A\big(\Op^A(\Psi)\big)
	\eeq
	where
\beq 
	\big(\mathbb{M}_1\cdot\mathbb{M}_2\big)_{\alpha,\alpha^*;\beta,\beta^*}\,:=\,\underset{(\gamma,\gamma^*)\in\Gamma\times\Gamma_*}{\sum}\big(\mathbb{M}_1\big)_{\alpha,\alpha^*;\gamma,\gamma^*}\big(\mathbb{M}_2\big)_{\gamma,\gamma^*;\beta,\beta^*}.
	\eeq 
	\end{subequations}
\end{proposition}
\begin{proof}
	Let us start with the definition of the left hand side in (\ref{F-prodMAOp}a) and write:
	\begin{align*}
		&\big(\mathbb{M}^A\big(\Op^A(\Phi)\Op^A(\Psi)\big)\big)_{\alpha,\alpha^*;\beta,\beta^*}\hspace*{-0.2cm}\\
		&\hspace*{2cm}=\big(\mathcal{G}^A_{\alpha,\alpha^*}\,,\,\Op^A(\Phi)\Op^A(\Psi)\mathcal{G}^A_{\beta,\beta^*}\big)_{L^2(\X)}=\big(\Op^A(\overline{\Phi})\mathcal{G}^A_{\alpha,\alpha^*}\,,\,\Op^A(\Psi)\mathcal{G}^A_{\beta,\beta^*}\big)_{L^2(\X)}\\
		&\hspace*{2cm}=\underset{(\gamma,\gamma^*)\in\Gamma\times\Gamma_*}{\sum}\overline{\big(\mathcal{G}^A_{\gamma,\gamma^*}\,,\,\Op^A(\overline{\Phi})\mathcal{G}^A_{\alpha,\alpha^*}\big)}_{L^2(\X)}\big(\mathcal{G}^A_{\gamma,\gamma^*}\,,\,\Op^A(\Psi)\mathcal{G}^A_{\alpha,\alpha^*}\big)_{L^2(\X)}\\
		&\hspace*{2cm}=\underset{(\gamma,\gamma^*)\in\Gamma\times\Gamma_*}{\sum}\big(\mathbb{M}^A\big(\Op^A(\Phi)\big)\big)_{\alpha,\alpha^*;\gamma,\gamma^*}\big(\mathbb{M}^A\big(\Op^A(\Psi)\big)\big)_{\gamma,\gamma^*;\beta,\beta^*}.
	\end{align*}
\end{proof}

\begin{remark}
	We notice that for any $\Phi\in S^0_0(\X^*,\X)$  we have the equalities:
	$$
	\big(\mathbb{M}^A\big(\Op^A(\overline{\Phi})\big)\big)_{\alpha,\alpha^*;\gamma,\gamma^*}\,=\,\big(\mathbb{M}^A\big(\Op^A(\Phi)^*\big)\big)_{\alpha,\alpha^*;\gamma,\gamma^*}\,=\,\overline{\big(\mathbb{M}^A\big(\Op^A(\Phi)\big)\big)}_{\gamma,\gamma^*;\alpha,\alpha^*}.
	$$
\end{remark}

\vspace{0.5cm}

\noindent Let us recall from \cite{MP-1,IMP-1} that  \textit{the \say{magnetic} Moyal product} $\sharp^B$ is defined by the equality:
$$
\Op^A\big(\phi\sharp^B\psi\big)\,:=\, \Op^A(\phi)\,\Op^A(\psi),\quad\forall(\phi,\psi)\in\mathscr{S}(\Xi)\times\mathscr{S}(\Xi)\,.
$$
 It is  given explicitly by the following  integral:
\beq\label{DF-mMprod}
\Big(\phi\sharp^B\psi\big)(X)=\pi^{-2d}\int_{\Xi\times\Xi}dY\,dY'\,e^{-2i(<\xi-\eta,x-y'>-<\xi-\eta',x-z>)}\,\omega^B_x(y,y')\,\phi(Y)\,\psi(Y')
\eeq
where  $\omega^B_x(y,y')$ is the exponential of $(-i)$ multiplied with the flux of $B$ through the triangle with vertices $x-y-y',x-y+y',x+y-y'$.\\
 As shown in \cite{MP-1,IMP-1}, this \say{magnetic} Moyal product may be extended as a composition law on a large class of tempered distributions on $\X\times\X^*$ that contains the H\"{o}rmander classes for all $p\in\R$.

By using  a direct combination of Propositions \ref{P-prodMAOp} and \ref{P-comp-rdec-offd-m} with Theorem \ref{T-m-Gabor-frame}, we get:
\begin{corollary}\label{coro1}
Given $(p,q)\in\mathbb{R}\times\mathbb{R}$ and $(\Phi,\Psi)\in S^p_0(\X\times \X^*)\times S^q_0(\X\times \X^*)$ we have that $\Phi\sharp^B\Psi\in S^{p+q}_0(\X\times \X^*)$.
\end{corollary}

\subsection{A magnetic version of the Calder\'on-Vaillancourt Theorem.}

Let us consider a symbol $F\in S^0_0(\X^*,\X)$ and a regular magnetic field $B$ obeying \eqref{dhc1}. In \cite{IMP-1} the following result is proven, using pseudo-differential calculus techniques:
\begin{theorem}\label{T-m-CV-Thm}
Under the  above assumptions, we have that $\Op^A(F)$ is bounded in  $L^2(\X)$. Moreover, there exist { $c(d)>0$,  $p(d)\in\mathbb{N}$,  and $N\geq 0$,  such that for all $F\in S^0_0(\X^*,\X)$ and any regular $B$ obeying \eqref{dhc1} we have 
$$
\big\|\Op^A(F)\big\|_{\mathcal L (L^2(\X))}\,\leq\,c(d)\big (1+|||B|||_N\big )\,\underset{|\alpha|\leq p(d)}{\max}\,\underset{|\beta|\leq p(d)}{\max}\,\underset{X\in\Xi}{\sup}\,\big|\big(\partial_x^\alpha\partial_\xi^\beta\,F\big)(X)\big|.
$$}
\end{theorem}

\begin{proof} Let us present here a simple proof of the boundedness of the operator $\Op^A(F)$ in $L^2(\X)$, based on the use of magnetic Gabor matrices.
Given any $f\in \mathscr{S}(\X)$ we have $\Op^A(F)f\in L^2(\X)$ and 

{
$$\big\|\Op^A(F)f\big\|_{L^2(\X)}=\sup_{g\in \mathscr{S}(\X),\, \| g\|_{L^2(\X)}=1} \big \vert \big (g,\Op^A(F)f\big )_{L^2(\X)}\big \vert .$$

From \eqref{dc5'} we have:
\begin{align*}
\big \vert \big (g,\Op^A(F)f\big )_{L^2(\X)}\big \vert 
&\leq \underset{(\gamma,\gamma^*),\, (\alpha,\alpha^*)\in\Gamma\times\Gamma_*}{\sum} \big |\big(\mathcal{G}^A_{\gamma,\gamma^*},g\big)_{L^2(\X)}\big|\,\big |\MAF_{\gamma,\gamma^*;\alpha,\alpha^*}\big |\,\big |\big(\mathcal{G}^A_{\alpha,\alpha^*},\,f\big)_{L^2(\X)}\big|.
\end{align*}
 The estimate \eqref{F-Mdecay} applied with $p=0$ implies that there exists a constant $C_d(F,B)$ depending on a finite number of seminorms of $F$ and $B$ such that $$\big |\MAF_{\gamma,\gamma^*;\alpha,\alpha^*}\big | \leq C_d(F,B) \, <\gamma-\alpha>^{-d-1} \, <\gamma^*-\alpha^*>^{-d-1}.$$ Applying the Schur test in $\ell^2(\Gamma\times \Gamma^*)$ and using the isometric property from Proposition \ref{P-Parseval-frame}(i) we obtain that 
\begin{align*}
\big \vert \big (g,\Op^A(F)f\big )_{L^2(\X)}\big \vert 
\leq C_d(F,B) \Vert g\Vert_{L^2(\X)}\, \Vert f\Vert_{L^2(\X)}
\end{align*}
and we are done, up to a density argument. }

\end{proof}

\subsection{On the Beals commutator criterion}

Finally let us now complete the result in \cite{CHP-3} and also prove the reciprocal statement for the Beals criterion. 

In order to state this criterion let us recall the \say{basic symbols}: 
\begin{itemize}
\item the position coordinates $Q_j:=\Op^A(q_j)=\Op^0(q_j)$ with $q_j(x,\xi):=x_j$ for $1\leq j\leq d$ 
\item  the \say{magnetic} momenta $P^A_j:=\Op^A(p_j)=\Op^0(p_j)-A_j$ with $p_j(x,\xi):=\xi_j$ for $1\leq j\leq d$.
\end{itemize}
 Let us notice that the symbols $q_j$ for $1\leq j\leq d$ are not H\"{o}rmander type symbols and that the above operators are continuous as operators in $\mathscr{S}(\X)$ and respectively in $\mathscr{S}^\prime(\X)$.
\begin{theorem}\label{T-recBeals}
If $\Phi\in S^0_0(\X\times \X^*)$, then $\Op^A(\Phi)$ defines a bounded linear operator on $L^2(\X)$ having bounded repeated commutators of the form $[L_1,[L_2,\ldots[L_N,T]\ldots]]$ for any $N\in\mathbb{N}$ and any family $\{L_1,\ldots,L_N\big\}$ (void if $N=0$)
with $L_m$ any of the basic observables $\big\{Q_1,\ldots,Q_d,P^A_1,\ldots,P^A_d\big\}$.\\
 Conversely, assume that a linear map $T:\mathscr{S}(\X)\mapsto \mathscr{S}'(\X)$ can be extended to a bounded operator on $L^2(\X)$, and  all its possible commutators as above have the same property; then $T$ is a magnetic pseudo-differential operator with a symbol of class $S^0_0(\X\times \X^*)$.
\end{theorem}
\begin{proof}
Let us prove the direct implication. Using \eqref{DF-mMprod} we have:
\begin{equation*}
\big[Q_j,\Op^A(\Phi)\big]=\Op^A\big(q_j\sharp^B\Phi-\Phi\sharp^Bq_j\big),
\end{equation*}
with 
\begin{align*}
&\big(q_j\sharp^B\Phi-\Phi\sharp^Bq_j\big)(X)\\
&\quad =\pi^{-2d} \int_{\Xi\times\Xi}dY\,dY'\,e^{-2i(<\xi-\eta,x-y'>-<\xi-\eta',x-y>)}\,\omega^B_x(y,y')\,\big(y_j\Phi(y',\eta')-y_j'\Phi(y,\eta)\big)\\
&\quad =i(\partial_{\xi_j}\Phi)(x,\xi)\,,
\end{align*}
and
\begin{equation*}
\big[P^A_j,\Op^A(\Phi)\big]=\Op^A\big(p_j\sharp^B\Phi-\Phi\sharp^Bp_j\big)\,,
\end{equation*}
with

\begin{align*}
&\big(p_j\sharp^B\Phi-\Phi\sharp^Bp_j\big)(x,\xi)\\
&=\pi^{-2d}\int_{\Xi\times\Xi}dY\,dY'\,e^{-2i(<\xi-\eta,x-y'>-<\xi-\eta',x-y>)}\,\omega^B_x(y,y')\,\big(\eta_j\Phi(y',\eta')-\eta_j'\Phi(y,\eta)\big)\\
&=-i\big(\partial_{x_j}\Phi\big)(x,\xi)\,+\,(2\pi)^{-d}\int_{\X}dv\int_{\X^*}d\eta\,e^{-i<\xi-\eta,v>}\Phi(x,\eta)\underset{1\leq k\leq d}{\sum}v_k\int_{-1/2}^{1/2}dr\,B_{kj}(x+rv)\\
&=-i\big(\partial_{x_j}\Phi\big)(x,\xi)\,+\,i(2\pi)^{-d}\int_{\X}dv\int_{\X^*}d\eta\,e^{-i<\xi-\eta,v>}\underset{1\leq k\leq d}{\sum}\big(\partial_{\eta_k}\Phi\big)(x,\eta)\int_{-1/2}^{1/2}dr\,B_{kj}(x+rv).
\end{align*}

We see that both above commutators have symbols in $S^0_0(\X\times\X^*)$, which remains true regardless how many commutators we perform afterwards. Then the Calder\'on-Vaillancourt theorem implies that these commutators can be extended to bounded operators on $L^2(\X)$.

The converse statement has been proved in \cite{CHP-3}, but now we can give a more transparent explanation of the strategy used there. The main idea was to first show that the matrix elements of $\mathbb{M}^A(T)$ obey an estimate like in \eqref{F-Mdecay}, and secondly, to  construct a symbol \say{by hand}, similarly with what we do here in Theorem \ref{T-m-Gabor-frame}(ii). 

The main idea behind the proof of \eqref{F-Mdecay} is the following: knowing that operators like $[Q_{j_1},[Q_{j_2},[....,[Q_{j_n},T]...]$ extend to bounded operators on $L^2(\X)$, the matrix element of $T$ must decay faster than any power of $<\alpha-\alpha'>$. Also, the boundedness of $[P^A_{j_1},[P^A_{j_2},[....,[P^A_{j_n},T]...]$ plus integration by parts, leads to fast decay in $<\alpha^*-\beta^*>$.  All details may be found in \cite{CHP-3}. 
\end{proof}

\subsection{Local Schatten-class properties}
{
For the non-magnetic case, a lot of results of this type are available in the literature \cite{A, R}. They not only give optimal decay conditions on the symbol, but also on its minimal regularity. Within our class of magnetic operators, we give a result which is close to be optimal even for the non-magnetic case:

\begin{theorem}\label{thm-trace}
  Let   $\Phi\in S^p_0(\X\times \X^*)$ with $p<-d$. Then for every $q>d$, the operators   $<\cdot>^{-q}\, \Op^A(\Phi)$ and $<\cdot>^{-q/2}\, \Op^A(\Phi)\, <\cdot>^{-q/2}$ are trace class. Moreover, if $p<-d/2$ and $r>d/2$, then $ \Op^A(\Phi)\, <\cdot>^{-r}$ is Hilbert-Schmidt. 
\end{theorem}
\begin{proof}
The integral kernel of $<\cdot>^{-q}\, \Op^A(\Phi)$ is 
    $$
	\underset{(\talpha,\tbeta)\in\tGamma\times\tGamma}{\sum}\mathbb{M}^A[\Op^A(\Phi)]_{\talpha,\tbeta}\, <x>^{-q}\mathcal{G}^A_{\talpha}(x)\, \overline{\mathcal{G}^A_{\tbeta}}(y),
	$$
hence this operator can be seen as a series of rank-one operators. There exists a constant $C$ such that the trace norm of these rank-one operators is bounded by $C <\alpha>^{-q}$ uniformly in $\alpha^*$ and $\tbeta=(\beta,\beta^*)$.  Hence $<\cdot>^{-q}\, \Op^A(\Phi)$ is trace class if we can prove that 
$$\underset{\alpha^*, \beta^*\in \Gamma^*}{\sum}\,\, \underset{\alpha,\beta\in \Gamma}{\sum}\,<\alpha>^{-q}\, \big | \mathbb{M}^A[\Op^A(\Phi)]_{\talpha,\tbeta}\big | \, <\infty.$$
Let us choose $n_1=n_2=d+1$ in \eqref{F-Mdecay}. Then we have 
\begin{align*}
& <\alpha>^{-q}\, \big | \mathbb{M}^A[\Op^A(\Phi)]_{\talpha,\tbeta}\big | \leq \\
&\qquad \leq C_{n_1n_2}(\Phi,B)\, <\beta-\alpha>^{-d-1}<\alpha>^{-q}<\alpha^*-\beta^*>^{-d-1} <\alpha^*+\beta^*>^{p}.
\end{align*}
Since $p<-d$ and $q>d$, the series is convergent. 

The proof for $<\cdot>^{-q/2} \Op^A(\Phi) <\cdot>^{-q/2}$ is similar; here we need to show 
$$\sum_{\alpha,\beta\in\Gamma} <\beta>^{-q/2}<\beta-\alpha>^{-d-1}<\alpha>^{-q/2}<\infty,$$
which again is a consequence of the Young inequality for convolutions. 
Finally, denoting by 
$ T$ the operator $\Op^A(\Phi)\, <\cdot>^{-r}$, we have $T^*T=<\cdot>^{-r}\, \Op^A(\overline{\Phi})\,\Op^A(\Phi)\,<\cdot>^{-r}$. The symbol of the product in the middle  belongs to $S_0^{2p}$ with $2p<-d$ while $r>d/2$, hence $T^*T$ is trace-class. 
\end{proof}

}

\end{document}